\documentclass{article}
\usepackage{amsmath}
\usepackage{amsfonts}
\usepackage{amssymb}
\usepackage{amsthm}
\usepackage{graphicx}
\usepackage{cleveref}
\usepackage{enumerate}
\usepackage{epstopdf}
\usepackage{fullpage}
\usepackage{color}

\newtheorem{Thm}{Theorem}[section]

\newtheorem{Lem}{Lemma}[section]

\newtheorem{Pro}{Proposition}[section]

\theoremstyle{definition}
\newtheorem{Def}[Thm]{Definition}

\theoremstyle{remark}
\newtheorem{Rem}{\bf{Remark}}[section]

\newcommand{\wh}{\widehat}
 
\newcommand{\EqD}{\overset{d}{=}}
\newcommand{\ConvD}{\overset{d}{\rightarrow}}
\newcommand{\ConvFDD}{\overset{f.d.d.}{\longrightarrow}}
\newcommand{\EqFDD}{\overset{f.d.d.}{=}}
\newcommand{\cl}{\mathcal}
\newcommand{\wt}{\widetilde}
\newcommand{\bb}{\mathbb}
\newcommand{\E }{\mathbb{E}}
\newcommand{\mbf}{\boldsymbol}

\newcommand{\B}{\mathrm{B}}
\newcommand{\Var}{\mathrm{Var}}

\numberwithin{equation}{section}

\title{Representations of Hermite processes using local time of intersecting stationary stable regenerative sets}
\author{Shuyang Bai}

\begin{document}

\maketitle

\begin{abstract}
Hermite processes are a class of self-similar processes with stationary increments. They often arise in limit theorems under long-range dependence. We derive new representations of Hermite processes with multiple Wiener-It\^o integrals, whose integrands involve the local time of intersecting stationary stable regenerative sets.  The proof relies on  an approximation of    regenerative sets and local times based on a scheme of random interval covering.
\end{abstract}

\textbf{MCS Classification (2010)}:  60G18.

\section{Introduction}\label{Sec:Intro}
Since the seminal works of Taqqu \cite{taqqu:1975:weak, taqqu:1979:convergence}  and Dobrushin and Major \cite{dobrushin:major:1979:non}, the class of processes called   Hermite processes have attracted considerable interest in probability and statistics.  A Hermite process, up to a multiplicative constant, is specified by two parameters: the \emph{order} $p\in \bb{Z}_+$ and the \emph{memory parameter} 
\begin{equation}\label{eq:beta}
\beta\in \left(1-\frac{1}{p},1\right).
\end{equation}
 A Hermite process  can be defined by any of its equivalent  representations. Here by equivalent representations, we mean   the  represented  processes share the same finite-dimensional distributions.   
  Two of the most well-known representations  are the \emph{time-domain representation} and the \emph{frequency-domain representation} in terms of  \emph{multiple Wiener-It\^o integrals} (see Section \ref{Sec:wi} below). The time-domain representation is given by  
\begin{equation}\label{eq:time domain}
Z_1(t)= a_{p,\beta} \int_{\bb{R}^p}'  \left(\int_0^t  \prod_{j=1}^p (s-x_j)^{\beta/2-1}_+ ds \right) W(dx_1)\ldots W(dx_p),\quad t\ge 0,
\end{equation}
where   $W$ is a Gaussian random measure on $\bb{R}$ with Lebesgue control measure,  the prime $'$ at the top of the integral sign indicates the exclusion of the diagonals $x_i=x_j$, $i\neq j$, in the $p$-tuple stochastic integral,  
 and 
\[
a_{p,\beta}=\left( \frac{(1-p(1-\beta)/2)(1-p(1-\beta))}{p! \mathrm{B}(\beta/2,1-\beta)^p} 
\right)^{1/2} 
\]
is a constant to ensure that $\Var[Z_1(1)]=1$, where $\B(\cdot,\cdot)$  is the beta function. The frequency-domain representation is given by  
\begin{equation}\label{eq:freq domain}
Z_2(t)=b_{p,\beta} \int_{\bb{R}^p}''\frac{e^{it(x_1+\ldots+x_p) }-1}{i(x_1+\ldots+x_k)} \prod_{j=1}^p |x_j|^{-\beta/2} \wh{W}(dx_1)\ldots \wh{W}(dx_p),\quad t\ge 0, 
\end{equation}
where $\wh{W}$ is a complex-valued Hermitian Gaussian random measure  \cite[Definition B.1.3]{pipiras:2017:long} with Lebesgue control measure, the double prime $''$ at the top of the integral sign indicates the exclusion of the hyper-diagonals $x_i=\pm x_j$, $i\neq j$, in the $p$-tuple stochastic integral,  and
\[
b_{p,\beta}=\left( \frac{(p(\beta-1)/2+1)(p(\beta-1)+1)}{p! [\Gamma(1-\beta)\sin (\beta\pi/2)]^p} 
\right)^{1/2}
\]
is a constant to ensure $\Var[Z_2(1)]=1$, where $\Gamma(\cdot)$ is  the gamma function. See \cite[Section 4.2]{pipiras:2017:long} for the derivation of the normalization constants $a_{p,\beta}$ and $b_{p,\beta}$.  It was shown in \cite{taqqu:1979:convergence} that $Z_1(t)$ and $Z_2(t)$ have the same finite-dimensional distributions and thus they represent the same process, which we denote as $Z(t)$. We shall  call such a process a  standard  Hermite process, where the word \emph{standard} corresponds to the normalization $\Var[Z(1)]=1$.

A Hermite process $Z(t)$   has stationary increments and is self-similar with Hurst index 
$H=1-p(1-\beta)/2\in (1/2,1)$,
 namely,  $(Z(ct))_{t\ge 0}$ and $(c^H Z(t))_{t\ge 0}$  have the same finite-dimensional distributions for any constant $c>0$.  In literature, $H$ is often used in place of $\beta$ to parameterize $Z(t)$, whereas   $\beta$ is a convenient choice for this paper.
When the order $p=1$, $Z(t)$ recovers a well-known Gaussian process:   fractional Brownian motion. When $p\ge 2$, the law of $Z(t)$ is non-Gaussian, and  if $p=2$, the process is also known as a Rosenblatt process  \cite{rosenblatt:1961:independence,taqqu:1975:weak}.  All standard  Hermite processes with   Hurst index $H$, regardless of the order, share  the same covariance structure as a standard fractional Brownian motion, that is,
\[
\E [Z(t_1) Z(t_2)]=\frac{1}{2}\left(|t_1|^{2H}+|t_2|^{2H}- |t_1-t_2|^{2H}\right), \quad t_1,t_2\ge 0.
\]

A Hermite process $Z(t)$ often  appears in, but not limited to, a limit theorem  of the form
\begin{equation}\label{eq:limit theorem}
\left(\frac{1}{A(N)}\sum_{n=1}^{\lfloor Nt \rfloor}X_n\right)_{t\ge 0} \Rightarrow (Z(t))_{t\ge 0}, \quad \text{as }N\rightarrow\infty,
\end{equation}
where $\Rightarrow$ stands for a suitable sense of weak convergence (e.g., convergence of finite-dimensional distributions, or weak convergence in Skorokhod space), $A(N)$ is a normalizing sequence regularly varying \cite{bingham:goldie:teugels:1989:regular} with index $H$ as $N\rightarrow\infty$,    $(X_n)$ is a stationary sequence with \emph{long-range dependence}, a notion often characterized by a slow power-law decay of the covariance of $(X_n)$. See, e.g,  \cite{dobrushin:major:1979:non, taqqu:1979:convergence, surgailis:1982:zones,ho:hsing:1997:limit}.  These limit theorems are often termed \emph{non-central limit theorems}, which have found numerous applications in statistical inference under long-range dependence. See, e.g., \cite{bai:2018:instability} and the references therein.

Some alternative representations  of a Hermite process   are known
besides the ones  in \eqref{eq:time domain} and \eqref{eq:freq domain}. Two other representations based on multiple Wiener-It\^o integrals  are  the  \emph{finite-time interval representation} \cite{tudor:2008:analysis,pipiras:taqqu:2010:regularization} and the \emph{positive half-axis representation} \cite{pipiras:taqqu:2010:regularization}. See also  \cite[Section 4.2]{pipiras:2017:long}. In addition, there is a representation involving multiple integral with respect to fractional Brownian motion \cite[Definition 7]{nourdin:2010:central}. See also  \cite[Section 3.1]{tudor:2013:analysis}.
Typically   to obtain a Hermite process as the weak limit, one needs to work with a suitable choice  among these different representations.
    
In this paper, we shall  provide  new multiple Wiener-It\^o integral  representations of   Hermite processes    of   different nature. These representations   involve  the local time of intersecting stationary stable regenerative sets (see Section \ref{Sec:Prelim} below). The reader may directly skip to  Theorem \ref{Thm:Herm Proc} below for a  glimpse. The discovery of such  representations is motivated by the recent works \cite{bai:2019:flow, bai:2019:multstable}. 
 The new representations also shed light on new mechanisms (e.g. \cite{bai:2019:flow}) which may generate Hermite processes as weak limits.

The rest of the paper is organized as follows.    Section  \ref{Sec:Prelim} prepares some necessary background.    Section \ref{Sec:Main} states the main results. The proofs of the results in Section \ref{Sec:Main} are collected in Section \ref{Sec:Pf}.

\section{Preliminaries}\label{Sec:Prelim}

\subsection{Multiple Wiener-It\^o integrals}\label{Sec:wi}
The information  recalled  below about Gaussian analysis and multiple Wiener-It\^o integrals can be found in \cite{janson:1997:gaussian}.    Let $(E,\cl{E},\mu)$ be a $\sigma$-finite measure space, and let $W$ be a  Gaussian (independently scattered) random measure on $E$ with control measure $\mu$ so that $\E W(A)W(B)=\mu(A\cap B)$ for $A,B\in \cl{E}$ with $\mu(A),\mu(B)<\infty$. 
Throughout   this paper, the underlying probability space which carries the randomness of Gaussian random measures is denoted by $(\Omega,\cl{F},\mathbb{P})$. 
Then for $g\in  L^2(E,\cl{E},\mu)$, the Wiener integral $I(g)=\int_E g(x) W(dx)$  can be defined, which forms a linear isometry between $L^2(E,\cl{E},\mu)$ and the Gaussian Hilbert space $H=\{I(g):\ g\in L^2(E,\cl{E},\mu)\}$. 

Let $H^{:p:}=\overline{\cl{P}}_p(H)\cap\left( \overline{\cl{P}}_{p-1}(H)^{\perp}\right)$ be  the $p$-th Wiener chaos of $H$, $p\ge 1$, where   $\overline{\cl{P}}_p(H)$ denotes the closure of the linear subspace of multivariate polynomials of elements of $H$ with degrees $p$ or lower, and $\perp$ in the superscript  denotes the orthogonal complement in $L^2(\Omega,\cl{F},\mathbb{P})$.  Then for $f\in L^2(E^p,\cl{E}^p,\mu^p)$, the $p$-tuple Wiener-It\^o integral 
\[
I_p(f)=\int_{E^p}'  f(x_1,\ldots,x_p)  W(dx_1)\ldots W(dx_p)
\]
can be defined (\cite[Theorem 7.26]{janson:1997:gaussian}).  In fact, $I_p:L^2(E^p,\cl{E}^p,\mu^p) \rightarrow H^{:p:}$ forms a bounded linear operator, and in particular, a \emph{linear isometry} from the subspace of symmetric functions of $L^2(E^p,\cl{E}^p,\mu^p/p!)$ to $H^{:p:}$. In addition, $I_p$ is characterized by the following property: for $f_1,\ldots,f_p\in L^2(E,\cl{E},\mu)$, we have
\begin{equation}\label{eq:mult int wick}
I_p(f_1\otimes\ldots\otimes f_p)=\int_{E^p}' f_1(x_1)\ldots f_p(x_p) W(dx_1)\ldots W(dx_p)=: I(f_1)\ldots I(f_p):,
\end{equation}
where for $\xi_1,\ldots,\xi_p\in H$, the notation $:\xi_1 \ldots \xi_p:$ stands for the Wick product, which is the  projection of the product $\xi_1\ldots \xi_p$   onto the $L^2$ subspace $H^{:p:}$. We note that in   literature, the construction of  $I_p(f)$ often starts with \eqref{eq:mult int wick}  for $f_i=1_{A_i}$ for disjoint $A_i\in \cl{F}$ with $\mu(A_i)<\infty$, so that $: I(f_1)\ldots I(f_p):$ is simply the product $W(A_1)\ldots W(A_p)$. Then the definition is extended to a general $f\in L^2(E^p,\cl{E}^p,\mu^p)$ by linearity and continuity given that $\mu$ is atomless. If $\mu$ has atoms, an extra step in the construction is needed (see e.g., P.42 of \cite{major:2014:multiple}). For generality, we shall use the construction of multiple Wiener-It\^o integrals in \cite[Chapter VII.2]{janson:1997:gaussian} without assuming that $\mu$ is atomless.

The following lemma is useful for changing the underlying measure spaces   in  multiple Wiener-It\^o integral representations of a process. 
\begin{Lem} 
\label{Lem:change space mult int}
Let $(E_j,\cl{E}_j,\mu_j)$, $j=1,2$, be $\sigma$-finite measure spaces.  Suppose  $W_j$ is a Gaussian random measure  defined on $E_j$ with control measure $\mu_j$, $j=1,2$.
Let $(U,\cl{H})$ be a measurable space and 
 suppose $f_j:E_j \rightarrow U  $, $j=1,2$, are measurable  and satisfy
 \begin{equation}\label{eq:equal in induced measure}
\mu_1 \circ f_1^{-1} = \mu_2 \circ f_2^{-1}=:\nu
 \end{equation}
 on $(U,\cl{H})$.
Let $T$ be an  index set and let  $g_t: U^p\rightarrow[-\infty,+\infty]$, $t\in T$,  be a family of  measurable functions   satisfying $g_t\circ f_j^{\otimes p} \in L^2(E^p_j,\cl{E}_j^p, \mu_j^p)$, $j=1,2$. Let
\begin{equation}\label{eq:Z_j(t)}
Z_j(t)=\int_{E_j^p}' g_t\left(f_j(x_1),\ldots,f_j(x_p)\right)  W_j(d x_1)\ldots  W_j(d x_p), \quad j=1,2.
\end{equation}
 Then the two processes $(Z_j(t))_{t\in T}$, $j=1,2$, have the same finite-dimensional distributions.
\end{Lem} 
\begin{proof} 

First, by Cram\'er-Wold and  linearity of the multiple integrals, it suffices to prove equality of marginal distributions at a single $t\in T$, and we thus simply set  $g_t=g$.   
Suppose first $g=h_1\otimes \ldots \otimes h_p$ where  all $h_1,\ldots,h_p\in L^2(U,\cl{H},\nu)$. Then by \eqref{eq:mult int wick},   the right-hand side of \eqref{eq:Z_j(t)} is equal to
\[
:\int_{E_j} h_1 \circ f_j (x) W_j(dx) \ldots \int_{E_j} h_p \circ f_j (x) W_j(dx):\ ,\quad j=1,2.
\]
So by joint Gaussianity, the conclusion follows from comparing the covariances: for $1\le i_1,i_2\le p$,
\begin{align*}
&\E \left[\int_{E_j} h_{i_1} \circ f_j (u) W_j(dx) \int_{E_j} h_{i_2} \circ f_j (x) W_j(dx)\right]\\=&\int h_{i_1}\circ f_j(x) \ h_{i_2}\circ f_j(x)\mu_j(dx) 
= \int_{U} h_{i_1}(u) h_{i_2}(u)\nu(du),\quad j=1,2,
\end{align*}
where we have used \eqref{eq:equal in induced measure} in the last equality. 

Similarly, using   linearity of the multiple integrals, the conclusion holds if $g$ is a finite linear combination of terms each of the form $h_1\otimes\ldots\otimes h_p$. At last, by the  linear isometry of the multiple integral,    it suffices to note that such linear combinations are dense in $L^2(U^p,\cl{H}^p,\nu^p)$.
\end{proof}

We mention that a similar discussion  as above carries over to the case where $W$ is replaced by a complex-valued    Gaussian random measure. See \cite[Chapter 7, Section 4]{janson:1997:gaussian}.

\subsection{Regenerative sets and local time functional}
Most of the information reviewed in this section about subordinators and regenerative sets can be found in \cite{bertoin:1999:subordinators}.

Recall that a process  $(\sigma(t))_{t\ge0}$ is said to be a  subordinator,   if it is  a non-decreasing   L\'evy process starting at the origin.  The Laplace exponent of $(\sigma(t))_{t\ge0}$ is given by $\E e^{-\lambda \sigma(t)}=\exp(-  t\Phi(\lambda)), \lambda\ge0$, which completely characterizes its law and satisfies  the L\'evy-Khintchine formula: 
\begin{equation}\label{eq:Levy Khin}
\Phi(\lambda)= d\lambda+\int_{(0,\infty)}(1-e^{-\lambda x})\Pi(dx),
\end{equation}
where the constant  $d\ge 0$ is  the drift and   $\Pi$ is the L\'evy measure on $(0,\infty)$ satisfying $\int_{(0,\infty)} (x\wedge 1) \nu(dx)<\infty$. The renewal measure $U$ of $(\sigma(t))_{t\ge 0}$ on $[0,\infty)$ is characterized by  $\int_{[0,\infty)}f(x)U(dx)=\E \int_0^\infty f(\sigma(t))dt$ for any   measurable function $f\ge 0$, which is related to the Laplace exponent through
\begin{equation}\label{eq:Lap Potential}
\int_{[0,\infty)}  e^{-\lambda x} U(dx)=\frac{1}{\Phi(\lambda)}, \quad \lambda\ge 0.
\end{equation}
 The Radon-Nikodym  derivative of $U$ with respect to the Lebesgue measure, if   exists, is called the renewal density.

 Let $\mbf{F}=\mbf{F}([0,\infty))$ denote the space of closed subsets of $[0,\infty)$ equipped with the Fell topology  \cite[Appendix C]{molchanov:2017:theory}. A random element $R$ taking value in  $\mbf{F}$ is said to be a \emph{regenerative set}, if $R$ has the same distribution  as  the closed range $\overline{\{\sigma(t):\ t\ge 0\}}$ where   $(\sigma(t))_{t\ge0}$ is a subordinator. Note that for any constant $c>0$, the time-scaled subordinator $(\sigma(ct))_{t\ge 0}$  and the original $(\sigma(t))_{t\ge 0}$ correspond to the same regenerative set. To make the correspondence between a regenerative set and a subordinator unique,  we shall always  assume the following normalization condition for the Laplace exponent of the subordinator:
\begin{equation}\label{eq:Laplace norm}
\Phi(1)=1.
\end{equation}

A regenerative set $R $ is said to be $\beta$-stable, $\beta\in (0,1)$, if the associated 
   subordinator $(\sigma (t))_{t\ge 0}$ is $\beta$-stable, namely, if the Laplace exponent $\Phi(\lambda)=\lambda^{\beta}$, which
   corresponds to  the   L\'evy measure
\begin{equation}\label{eq:Pi}
\Pi  (dx)=\frac{ \beta}{\Gamma(1-\beta)}x^{-1-\beta} 1_{(0,\infty)}(x) dx.
\end{equation}

A random closed set $F$ in  $\mbf{F}$     is said to be stationary  if 
\begin{equation}\label{eq:tau_x}
\tau_x (F):=F\cap [x,\infty)-x\EqD F
\end{equation} 
for any $x\ge 0$. A $\beta$-stable regenerative set $R$ itself is not  stationary. However,    stationarity in a generalized sense  can be obtained  by a shifted $R$ as follows.
\begin{Pro}\label{Pro:stat set}
Let $\bb{P}_R$ be the distribution of  a $\beta$-stable regenerative set on $\mbf{F}$ and let $\pi_V$ be a Borel measure  on $[0,\infty)$ with $\pi_V(dv)=(1-\beta) v^{-\beta} dv$, $v\in (0,\infty)$.  
There exists a $\sigma$-finite infinite-measure    space $(E,\cl{E},\mu)$ and a measurable mapping $ \overline{R} :E\rightarrow\mbf{F}$, such that  
\[
\mu (  \overline{R}   \in \cdot )=(\bb{P}_R\times \pi_V)(F+v\in \cdot ),
\]
where the right-hand side is understood as the push-forward  measure of $\bb{P}_R\times \pi_V$ under the mapping $(F,v)\mapsto  F+v$.  Furthermore, $ \overline{R} $ is stationary in the sense of
\[
\mu(\tau_x  \overline{R}  \in \cdot )= \mu(  \overline{R} \in \cdot ),
\]
where $\tau_x$ is as in \eqref{eq:tau_x}.
\end{Pro} 
\begin{Rem}
The proposition says that if $V$ is an improper  random variable   governed by  the infinite law $\pi_V$  independent of the $\beta$-stable regenerative set $R$, then the resulting shifted improper   random set $R+V$ is stationary. We use ``improper'' to mean that the distribution may be an infinite measure throughout this paper.  The proposition follows from \cite[Proposition 4.1(c)]{lacaux:2016:time} (note that their $1-\beta$ corresponds to our $\beta$). It can also be obtained by restricting a stationary $\beta$-stable regenerative set on $\bb{R}$ constructed in \cite{fitzsimmons:1988:stationary}  to $[0,\infty)$.  
Note that the improper distribution $\mu (  \overline{R}   \in \cdot )$   stays unchanged if $\pi_V$ is replaced by  a positive constant  multiple of it.    This follows from the self-similarity property: $cR\EqD R$, $c>0$, of a $\beta$-stable regenerative set $R$.
\end{Rem}
\begin{Def}\label{Def:stat set}
We call the improper random closed set $ \overline{R} $ in Proposition \ref{Pro:stat set} a stationary $\beta$-stable regenerative set.
\end{Def}

Next, we recall the local time functionals introduced by \cite{kingman:1973:intrinsic}. For $\beta\in (0,1)$,
define
\[ 
L^{(\beta)}: {F}\rightarrow [0,\infty], ~ L^{(\beta)}(F) :=\limsup_{n\rightarrow\infty}\frac{\lambda\left(F+[- \frac{1}{2n},\frac{1}{2n}]\right)}{ l_{\beta}(n)},
\] 
where $\lambda$ is the Lebesgue measure, 
$F+[- 1/{2n},1/{2n}]=\cup_{x\in F}[x-1/{2n},x+1/{2n}]$,
 and the   normalization sequence
\[ 
l_{\beta}(n)=\int_{0}^{1/n} \Pi((x,\infty))dx= \frac{   n^{\beta-1}}{\Gamma(2-\beta)},
\] 
where $\Pi$ is as in \eqref{eq:Pi}.
We then define 
\begin{equation}\label{eq:L_t}
L_t^{(\beta)}(F):=L^{(\beta)}(F\cap [0,t]), \quad t\ge 0.
\end{equation}
By \cite[Lemma 2.1]{bai:2019:multstable},
  $L_t^{(\beta)}$ is  a measurable  mapping from $\mbf{F}$ to $[0,\infty]$ for each $t\ge 0$. Furthermore,  \cite[Theorem 3]{kingman:1973:intrinsic} entails that if $R$ is a $\beta$-stable regenerative set, then the process $(L_t^{(\beta)}(R ))_{t\ge 0}$  has the same finite-dimensional distributions as the local time process associated with $R$, which is an inverse $\beta$-stable subordinator.

\section{Main results}\label{Sec:Main}

We are now ready to state our main results.  Set 
\begin{equation}\label{eq:beta_p}
\beta_p:=(\beta-1)p+1\in (0,1).
\end{equation}
The range above should be compared to \eqref{eq:beta}.  Recall the standard Hermite process with memory parameter $\beta$ and order $p$ as defined in \eqref{eq:time domain} or \eqref{eq:freq domain}.
\begin{Thm}\label{Thm:Herm Proc} Let      $L_t =L_t^{(\beta_p)}$ be Kingman's local time functional as  in \eqref{eq:L_t} associated with a $\beta_p$-stable regenerative set.
\begin{enumerate}[(a)]
\item Let $ \overline{R} :E \rightarrow\mbf{F}$ be a stationary $\beta$-stable regenerative set defined on a $\sigma$-finite    measure space $(E,\cl{E},\mu)$ as in Definition \ref{Def:stat set}. 
Suppose $W$ is a Gaussian random measure on $E$ with control measure $\mu$.
Then the process
\begin{equation}\label{eq:new rep 1}
Z(t)=c_{p,\beta}\int_{E^p}' L_t \left(\bigcap_{i=1}^p   \overline{R} (x_i)\right) W(d x_1)\ldots W(dx_p)
\end{equation}
has the same finite-dimensional distributions as the standard Hermite process with memory parameter $\beta$ and order $p$, where
\begin{equation}\label{eq:c_{p,beta}}
c_{p,\beta}=\left(\frac{\Gamma(\beta_p) \Gamma(\beta_p+2)
 }{ 2p! \Gamma(\beta)^{p } \Gamma(2-\beta)^{p }}\right)^{1/2}. 
\end{equation}
\item Let $(\Omega' , \cl{F}'  , \mathbb{P}' )$ be a probability space (Note that $(\Omega' , \cl{F}'  , \mathbb{P}' )$ is different from the   probability space $(\Omega,\cl{F},\mathbb{P})$ carrying the randomness of the Gaussian random measure).  Let $R :\Omega'\rightarrow\mbf{F}$ be a $\beta$-stable regenerative set and let $V :\Omega'\rightarrow [0,T]$, $T>0$, be a random variable independent of $R$ satisfying $\mathbb{P}'(V \le v)= T^{\beta-1} v^{1-\beta}$, $v\in [0,T]$.  Suppose $W_T$ is a Gaussian random measure on $\Omega'$ with control measure $\mathbb{P}'$.
Then
 \begin{align}\label{eq:new rep 2}
Z(t)=c_{p,\beta,T}\int_{(\Omega')^p}' L_t  \left(\bigcap_{i=1}^p (R (\omega_i') +V (\omega_i'))\right)  W_T(d\omega_1')\ldots W_T(d\omega_p')   
\end{align}
  has the same finite-dimensional distributions as a standard Hermite process  with memory parameter $\beta$ and order $p$ over the   interval $[0,T]$, where $c_{p,\beta,T}=T^{p(1-\beta)/2}c_{p,\beta}$.
\end{enumerate}
\end{Thm}
 
The theorem is proved in Section \ref{Sec:Pf} below.

\begin{Rem}
  In view of  Lemma \ref{Lem:change space mult int},  the finite-dimensional distributions of the process $(Z(t))_{t\ge 0}$ do not depend on choice of the   measure spaces $(E,\cl{E},\mu)$ and $(\Omega',\cl{F}',\mathbb{P}')$.   
 Notation-wise simple representations  can be obtained when the measure spaces are chosen canonically. For example in (a), if one sets $(E,\cl{E},\mu)=(\mbf{F},\cl{B},\cl{L}_\beta)$, where $\cl{B}$ is the Borel $\sigma$-field induced by the Fell topology on $\mbf{F}$ and $\cl{L}_\beta$ denotes the infinite law of a stationary $\beta$-stable regenerative set on $\mbf{F}$, then \eqref{eq:new rep 1} reduces to
\[
c_{p,\beta}\int_{\mbf{F}^p}'  L_t \left(\bigcap_{i=1}^p   \overline{R}_i\right) W(d  \overline{R}_i)\ldots W(d \overline{R}_i),
\] 
where $\overline{R}_i$, $i=1,\ldots,p$, are understood as the identity maps on $\mbf{F}$, and  $W$ is a Gaussian random measure on $\mbf{F}$ with control measure $\cl{L}_\beta$.  For a similar simplification of the representation in  (b), see \eqref{eq:rep interm} below.
\end{Rem}

\begin{Rem}
Note that even in the case $p=1$, Theorem \ref{Thm:Herm Proc} provides new representations for fractional Brownian motion whose Hurst index $H\in (0,1/2)$. In this case, by Gaussianity, the theorem can be established by computing the covariance directly using \cite[Theorem 2.2]{bai:2019:multstable} and  the $L^2$ linear isometry of single stochastic integrals. The proof of Theorem \ref{Thm:Herm Proc} given below, on the other hand, applies to any   $p\in \bb{Z}_+$. 

In addition, it is possible to use \cite[Theorem 2.2]{bai:2019:multstable} and the diagram formula for multiple integrals (\cite[Proposition 4.3.4]{pipiras:2017:long}) to compute the joint moments $\E [Z(t_1)\ldots Z(t_r)]$, $t_1,\ldots,t_r\ge 0$ and identify them with those of standard Hermite processes  (see \cite[Proposition 4.4.1]{pipiras:2017:long}).  This suffices to establish the claims of Theorem \ref{Thm:Herm Proc} when $p\le 2$, but fails to conclude for $p\ge 3$ since  moment determinancy ceases to hold  in this case \cite{slud:1993:moment}.
\end{Rem}

\begin{Rem}
The representations found in this theorem are motivated from \cite{bai:2019:multstable}, where a process is defined similarly as in   \eqref{eq:new rep 1}  or \eqref{eq:new rep 2} but with the Gaussian random measures   replaced by  $\alpha$-stable ones, $\alpha\in (0,2)$.  
  The representations in \eqref{eq:new rep 1} and \eqref{eq:new rep 2}   suggest the possibility of new types of limit theorems leading to Hermite processes. See for example \cite{bai:2019:flow}.
\end{Rem}

\section{Proofs}\label{Sec:Pf}

Throughout $c$ and $c_i$ denote constants whose values may change from line to line. We shall make use of the following  relations about beta and gamma functions: for $a,b>0$ and $x<y$,
\begin{equation}\label{eq:beta gamma}
\int_{y}^x  (u-y)^{a-1}(x-u)^{b-1} du=\B(a,b)(y-x)^{a+b-1},\quad \B(a,b)=\frac{\Gamma(a)\Gamma(b)}{\Gamma(a+b)}.
\end{equation}

\subsection{Constructions of regenerative sets and local times by  random covering}\label{Sec:cover}

The    proof of Theorem \ref{Thm:Herm Proc} uses  a   construction of a regenerative set as the set left uncovered by random intervals due to \cite{fitzsimmons:1985:set}, as well as a related construction of local time  which originates  from \cite{bertoin:2000:two}.   Similar constructions are used   in \cite{bai:2019:multstable}.

Suppose on the probability space $(H,\cl{H},\mathbb{P}_\cl{N})$,   a Poisson point process  $\cl{N}=\sum_{\ell\ge 1} \delta_{\{y_\ell,z_\ell\}}$   on $[0,\infty)^2$ is defined with intensity measure $(1-\beta)  dy z^{-2}dz$, $\beta\in (0,1)$, where $(\{y_\ell,z_\ell\})_{\ell\ge 1}$ is a measurable enumeration  of the points of $\cl{N}$. For $\epsilon\ge 0$, define 
\begin{equation}\label{eq:R_eps}
R_\epsilon  = \bigcap_{\ell:\ z_\ell\ge \epsilon}( y_\ell,y_\ell+z_\ell)^c,
\end{equation}
where the complement is with respect to $[0,\infty)$. In other words, $R_\epsilon $ is the subset of $[0,\infty)$  left uncovered by the collection of (possibly overlapping) open intervals $\{ (y_\ell,y_\ell+z_\ell):\ z_\ell\ge \epsilon\}$, where the left boundary $y_\ell$ and the width   $z_\ell$ are governed by $\cl{N}$.
In view of \cite[Example 1]{fitzsimmons:1985:set}, each $R_\epsilon $, $\epsilon\ge 0$, is a regenerative set on $[0,\infty)$, and in particular, $R_0$ is a $\beta$-stable regenerative set. It is also known that the subordinator associated with $R_\epsilon$ has a positive drift $d_\epsilon$  when $\epsilon>0$.    In contrast to  $R_0$, the L\'evy measure $\Pi_\epsilon$ and the drift $d_\epsilon$    of $R_\epsilon$, $\epsilon>0$, do not     have   simple expressions, although the exact expressions are not needed for our purpose.  While other constructions  of regenerative sets  are possible, the  random covering scheme seems efficient in providing tractable  approximation of the  intersecting  regenerative sets and the local times encountered  in the proof of  Theorem \ref{Thm:Herm Proc}.

By the calculation in the proof of Lemma 2.5 of \cite{bai:2019:multstable}, we have for $\epsilon>0$,
\begin{equation}\label{eq:point hit R_epsilon}
p_\epsilon(x):=\mathbb{P}_\cl{N}(x\in  R_{\epsilon})  =    e^{(\beta-1)x/\epsilon}1_{\{0\le x\le \epsilon\}}+ \left(\frac{\epsilon}{e}\right)^{1-\beta} x^{\beta-1}1_{\{x>\epsilon\}},\quad x\ge 0.
\end{equation} 
Let $u_\epsilon(x)$   be  the renewal density of the subordinator associated with $R_\epsilon$.
By \cite[Propositions 1.7]{bertoin:1999:subordinators}, if $\epsilon>0$, 
\begin{equation}\label{eq:renewal density}
u_\epsilon(x)=  d_{\epsilon}^{-1}p_\epsilon(x)=:d_{\epsilon}^{-1}\left(\frac{\epsilon}{e}\right)^{1-\beta}f_\epsilon(x), \quad x>0.  
\end{equation}
It is elementary to verify that  as $\epsilon\downarrow 0$,
\begin{equation}\label{eq:f_epsilon}
f_\epsilon(x)= (e^{x/\epsilon-1}\epsilon)^{\beta-1}1_{\{0\le x\le \epsilon\}}+ x^{\beta-1}1_{\{x>\epsilon\}}\uparrow x^{\beta-1},\quad x>0.
\end{equation}
Note that
 $\int_{0}^\infty e^{-x} u_\epsilon(x)dx=1$ due to normalization requirement \eqref{eq:Laplace norm} via \eqref{eq:Lap Potential}. Combining this fact with \eqref{eq:f_epsilon} and the monotone convergence theorem, we have $\epsilon\rightarrow 0$,
\begin{equation}\label{eq:d_eps}
d_\epsilon\sim  \Gamma(\beta)\left(\frac{\epsilon}{e}\right)^{1-\beta}  \ \text{ as }\epsilon\rightarrow 0.
\end{equation}
Note that $d_0=0$  since a $\beta$-stable subordinator has no drift.  
Define for all $\epsilon\ge 0$ the measures
\begin{equation}\label{eq:pi_epsilon}
\pi_\epsilon(dx)=d_\epsilon \delta_0 + \overline{\Pi}_\epsilon (x)dx,\ x\ge 0,\quad  \text{and }\quad  \wt{\pi}_\epsilon(\cdot )=\frac{\pi_\epsilon(\cdot\cap [0,1])}{\pi_\epsilon([0,1])},
\end{equation}
where $\overline{\Pi}_{\epsilon}(x)=\Pi_\epsilon((x,\infty))$ is the  tail of the  L\'evy measure $\Pi_\epsilon$ corresponding to $R_\epsilon $, and $\delta_0$ is the delta measure with a unit mass at $0$. Note that $\Pi_0$ is equal to $\Pi$ in \eqref{eq:Pi} and $\wt{\pi}_0=\pi_V$ in Proposition \ref{Pro:stat set}  when restricted to $[0,1]$. Note for each $\epsilon\ge 0$, 
\[
\int_0^y u_\epsilon(x)dx\asymp   y^{\beta}, \quad y>0,
\] 
where $h_1(y)\asymp h_2(y)$ means $ c_1 h_2(y)\le h_1(y)\le c_2 h_2(y)$ for some constants $0<c_1<c_2$. So 
in view of \cite[Proposition 1.4]{bertoin:1999:subordinators}, 
we have  
\[
\int_0^y \overline{\Pi}_{\epsilon}(x)dx \asymp y^{1-\beta}, \quad y>0.
\]
Hence each $\pi_\epsilon$ is a $\sigma$-finite infinite measure on $[0,\infty)$.

Enriching the  space $(H,\cl{H})$  if necessary, suppose $\nu$  is a $\sigma$-finite measure on $\cl{H}$, and  suppose $U_\epsilon:H\rightarrow[0,\infty)$ is an improper random variable   satisfying \begin{equation}\label{eq:R_eps U_eps}
\nu((R_\epsilon,U_\epsilon)\in\cdot)=(\bb{P}_{R_\epsilon}\times \pi_\epsilon)(\cdot),
\end{equation} 
where $\bb{P}_{R_\epsilon}$ is the distribution of $R_\epsilon$ on $\mbf{F}$ and $\pi_\epsilon$ is as in \eqref{eq:pi_epsilon}.    Then in view of \cite{fitzsimmons:1988:stationary}, the improper random set $\overline{R}_\epsilon: =R_\epsilon+U_\epsilon$ is   stationary   in the sense of
\begin{equation}\label{eq:stat improper eps}
\nu(\tau_x  \overline{R}_\epsilon  \in \cdot )= \nu(  \overline{R}_\epsilon \in \cdot ),
\end{equation}
where $\tau_x$ is the shift  as in \eqref{eq:tau_x}. 
 
The next result enables a key coupling argument in the proof of Theorem \ref{Thm:Herm Proc}.
\begin{Lem}\label{Lem:V_eps conv}
We have the  convergence in total variation distance: 
\[\|\wt{\pi}_\epsilon- \wt{\pi}_0\|_{TV}:=\sup\{|\wt{\pi}_\epsilon(A)-\wt{\pi}_0(A)|:\ A\in \cl{B}([0,1])\}\rightarrow 0\] as $\epsilon\rightarrow 0$
\end{Lem}
\begin{proof}
We first show  as $\epsilon\rightarrow 0$,
\begin{equation}\label{eq:conv pi total}
 \pi_\epsilon([0,1])\rightarrow \pi_0([0,1]).
\end{equation}
 By   \cite[Equation (1.11) and Proposition (3.9)]{fitzsimmons:1985:intersections} as $\epsilon\rightarrow 0$,  the Laplace exponent uniquely associated with $R_{\epsilon}$  satisfying \eqref{eq:Laplace norm} converges  pointwise to the Laplace exponent uniquely associated with a $\beta$-stable regenerative set. This, by \cite[Theorem 15.15(ii)]{kallenberg:2002:foundations}, further  implies that as $\epsilon\rightarrow 0$, we have the following weak convergence of measures:
\begin{align}
\wt{\nu}_\epsilon(dx):= d_\epsilon \delta_0 +(1-e^{-x})\Pi_{\epsilon}(dx) ~ \ConvD ~ &(1-e^{-x})\Pi_0(dx)=(1-e^{-x}) \frac{\beta}{\Gamma(1-\beta)}  x^{-\beta-1} dx=:\wt{\nu}_0(dx),\label{eq:nu conv distr}
\end{align} 
where $\wt{\nu}_\epsilon$, $\epsilon\ge0$, are probability measures on $[0,\infty)$ due to \eqref{eq:Levy Khin} and  \eqref{eq:Laplace norm}.
Next by Fubini,  
\begin{align}\label{eq:int pi eps}
\pi_\epsilon([0,1])=d_\epsilon+\int_0^1 \overline{\Pi}_\epsilon(x)dx=d_\epsilon+\int_{(0,\infty)}(1\wedge x)\Pi_\epsilon(dx) = \int_{[0,\infty)}  h(x)  \wt{\nu}_\epsilon(dx),
\end{align}
where
$
h(x):=\frac{1\wedge x}{1-e^{-x}} 
$ for $x>0$ and $h(0):=1$ is a bounded continuous function on $[0,\infty)$. Hence by the weak convergence in \eqref{eq:nu conv distr}, as $\epsilon \rightarrow 0$, we have
\begin{equation}\label{eq:int h_y}
\int_{[0,\infty)} h(x) \wt{\nu}_\epsilon(dx)\rightarrow \int_{[0,\infty)} h(x) \wt{\nu}_0(dx) =\int_0^1 \overline{\Pi}_0(x)dx=\pi_0([0,1])=\frac{1}{\Gamma(2-\beta)}.
\end{equation} 
Combining   \eqref{eq:int pi eps} and \eqref{eq:int h_y}, we obtain \eqref{eq:conv pi total}.
 
Now to conclude the proof, it suffices to show  
\begin{equation}\label{eq:conv total final goal}
\sup\{| \pi_\epsilon(A)-\pi_0(A)|:\ A\in \cl{B}([0,1])\}\rightarrow 0 
\end{equation}
as $\epsilon\rightarrow 0$.  Indeed  for $A\in \cl{B}([0,1])$,
\begin{align}\label{eq:bound total var}
 |\pi_\epsilon(A)-\pi_0(A)|= \left| d_\epsilon \delta_0(A)  + \int_{A\cap (0,1]} \overline{\Pi}_\epsilon(x) -  \overline{\Pi}_0(x) dx\right| \le  d_\epsilon+ \int_{[0,1]} |\overline{\Pi}_\epsilon(x) -  \overline{\Pi}_0(x)|dx.
\end{align}
Note that $d_\epsilon\rightarrow 0$ by \eqref{eq:d_eps}.  
On the other hand,  \cite[Proposition (3.9)]{fitzsimmons:1985:intersections} and \cite[Lemma 15.14(ii)]{kallenberg:2002:foundations} imply       as $\epsilon\rightarrow 0$ that
\begin{align*}
\overline{\Pi}_\epsilon(x) \rightarrow  \overline{\Pi}_0(x),\quad x>0.
\end{align*}
In addition by   \eqref{eq:int pi eps}, \eqref{eq:int h_y} and the fact $d_\epsilon\rightarrow 0$, we have as $\epsilon\rightarrow 0$
\[
\int_0^1 \overline{\Pi}_\epsilon(x)dx \rightarrow \int_0^1 \overline{\Pi}_0(x)dx.
\]
Therefore the second term in the bound in \eqref{eq:bound total var} tends to $0$  as $\epsilon\rightarrow 0$ by Scheff\'e's lemma (e.g, \cite[Item 5.10]{williams:1991:probability}).
So \eqref{eq:conv total final goal} follows. 
\end{proof}
 
Next we turn to the construction of the local time based on the covering scheme introduced above.
 Suppose $R_\epsilon$, $\epsilon\ge 0$, are as in \eqref{eq:R_eps} with the point process $\cl{N}$ defined on a probability space $(H,\cl{H},\mathbb{P}_\cl{N})$. In view of Lemma \ref{Lem:V_eps conv} and a well-known coupling characterization of total variation distance (e.g., \cite[Theorem 2.1]{sethuraman:2002:some}), there exist  random variables $V_\epsilon\ge 0$, $\epsilon\ge 0$,  defined on a probability space $(\Theta,\cl{G},\mathbb{P}_V)$, so that $\mathbb{P}_V(V_\epsilon\in\cdot)=\wt{\pi}_\epsilon(\cdot)$,  and as $\epsilon\rightarrow 0$,
\begin{equation}\label{eq:tv coupling}
\mathbb{P}_V(V_\epsilon = V_0)\rightarrow 1.
\end{equation}

\begin{Lem}\label{Lem:local from cover} 
\begin{enumerate}[(a)]
 For   $0\le  t\le 1$, set  
\begin{equation}\label{eq:L_t^eps}
L_t^{(\epsilon)}(\mbf{h},\mbf{\theta})=\frac{1}{\Gamma(\beta_p)} \left(\frac{\epsilon}{e}\right)^{\beta_p-1} \int_0^t 1\left\{x\in  \bigcap_{i=1}^p \left(R_\epsilon(h_i)+V_{\epsilon}(\theta_i)\right)  \right\} dx,
\end{equation}
where $\mbf{h}:=(h_1,\ldots,h_p)\in H^p$, $\mbf{\theta}:=(\theta_1,\ldots,\theta_p)\in \Theta^p$ and $\beta_p$ is as in \eqref{eq:beta_p}.
\item 
 As $\epsilon =1/n \rightarrow 0$, $n\in \bb{Z}_+$, $L_{t}^{(\epsilon)}$ converges  in $L^r=L^r\big(H^p\times \Theta^p,\cl{H}^{p}\times \cl{G}^p , \mathbb{P}_\cl{N}^p\times \mathbb{P}_V^p\big)$  to a limit $L_t^*$, $r\in \bb{Z}_+$. 
\item Let $\E'$ denote integration with respect to $ \mathbb{P}_\cl{N}^p\times \mathbb{P}_V^p$. Then for $r\in \bb{Z}_+$
\begin{equation}\label{eq:Delta moment}
\E' [(L_{t}^*)^r] =  \frac{r! \Gamma(2-\beta)^p\Gamma(\beta)
^p}{\Gamma((r-1)\beta_p+2)\Gamma(\beta_p)}    \  t^{(r-1)\beta_p-1}.
\end{equation}
\item  We have
\begin{equation}\label{eq:L*=L}
L_t^*(\mbf{h},\mbf{\theta})= L_t\left(\bigcap_{i=1}^p R_0(h_i ) +V_0(\theta_{i})\right)  \quad \mathbb{P}_\cl{N}^p\times \mathbb{P}_V^p\text{-a.e.},
\end{equation}
 where $L_t=L_t^{(\beta_p)}$ is the local time functional as in  \eqref{eq:L_t}.
\end{enumerate}
\end{Lem}
\begin{Rem}
Results similar to Lemma \ref{Lem:local from cover} above were  established in \cite{bai:2019:multstable} but with $V_\epsilon(\theta_i)$ and $V_0(\theta_i)$ replaced by fixed constants $v_i$, $i=1,\ldots,p$.
\end{Rem}

We need the following preparation for the proof of the lemma. \emph{Throughout an integral $\int_a^b \cdot \ dx$ is understood as zero if $a\ge b$. }
\begin{Lem}
 For   $\delta,\eta\ge 0 $,   $\mbf{v}=(v_1,\ldots,v_p)\in [0,1]^p$ and $\mbf{h}=(h_1,\ldots,h_p)\in H^p$, define
\[
\Delta_{\delta,\eta}^{(\epsilon)}(\mbf{h};\mbf{v})=\frac{1}{\Gamma(\beta_p)} \left(\frac{\epsilon}{e}\right)^{\beta_p-1} \int_\delta^\eta 1\left\{x\in  \bigcap_{i=1}^p (R_\epsilon(h_i)+v_i) \right\} dx
\]
Let $\E_\cl{N}$  denote integration with respect to $\mathbb{P}_\cl{N}^p$ on $H^p$ and suppose $r\in \bb{Z}_+$.
Then   
 \begin{align}
 \E_{\cl{N}}\left[ \Delta_{\delta,\eta}^{(\epsilon)}( \cdot ;\mbf{v})^r\right]=&\frac{r!}{\Gamma(\beta_p)^r}\int_{\delta<x_1<\ldots<x_r<\eta}  \left(\prod_{i=1}^p f_\epsilon(x_1-v_i) \right) f_\epsilon(x_2-x_1)^{p} \ldots f_{\epsilon}(x_r-x_{r-1})^{p}   d\mbf{x} \label{eq:Delta moment pre}\\ \le & c  (\eta-\delta)_+^{r\beta_p},\label{eq:bound Delta moment}
 \end{align}
 where $f_\epsilon(x)$ is as in \eqref{eq:f_epsilon}  if $x\ge 0$,   $f_\epsilon(x):=0$ for $x<0$,
and $c>0$ is constant which does not depend on $\epsilon$, $\delta$, $\eta$ or $\mbf{v}$.
\end{Lem}

\begin{proof} 
Suppose   $\delta< \eta$, otherwise the conclusion is trivial.
Let $p_\epsilon(x)$    be as in \eqref{eq:point hit R_epsilon} when $x\ge 0$ and   set $p_\epsilon(x)=0$ if $x<0$. 
 For $0\le x_1<\ldots<x_r$ and $\epsilon>0$, we have
\begin{align*}
\mathbb{P}_\cl{N}^p\left(\left\{\mbf{h}\in H^p:\ \{x_1,\ldots,x_r\}\subset  \bigcap_{i=1}^p (R_\epsilon(h_i)+v_i)\right\}\right)&= \prod_{i=1}^p \mathbb{P}_\cl{N}(\{h\in H:\ \{x_1,\ldots,x_r\}\subset (R_\epsilon(h)+v_i)\})\\
&=\prod_{i=1}^p \Big( p_\epsilon(x_1-v_i)p_\epsilon(x_2-x_1) \ldots p_{\epsilon}(x_r-x_{r-1})  \Big),
\end{align*}
where for the last equality, we have used the regenerative property of $R_\epsilon$ \cite[Equation (6)]{fitzsimmons:1985:set}. See also the proof of   \cite[Lemma 2.5]{bai:2019:multstable}. Then \eqref{eq:Delta moment pre} follows from
    Fubini, the relation $f_\epsilon(x)=  (\epsilon/e)^{\beta-1}p_\epsilon(x)$ (see \eqref{eq:renewal density}) and a symmetry in the multiple integral.
    
     Next, from \eqref{eq:f_epsilon} we have for any $\epsilon>0$,
\begin{equation}\label{eq:bound u_epsilon}
f_\epsilon(x)\le   x^{\beta-1}, \quad x>0,
\end{equation} 
  Hence  for some constants $c_1,c_2>0$ free of $\epsilon$, $\delta$, $\eta$ or $\mbf{v}$ (recall $\beta_p=p(\beta-1)+1\in (0,1)$), 
\begin{align*}
\E_{\cl{N}} \Delta_{\delta,\eta}^{(\epsilon)}( \cdot ;\mbf{v})^r &\le c_1\int_{\delta<x_1<\ldots<x_r<\eta}  \left(\prod_{i=1}^p  (x_1-v_i)_+^{\beta-1} \right)  (x_2-x_1)^{\beta_p-1} \ldots  (x_r-x_{r-1})^{\beta_p-1}   d\mbf{x}\\
&= c_2 \int_{\delta}^{\eta}  \left(\prod_{i=1}^p  (x_1-v_i)_+^{\beta-1} \right)  (\eta-x_1)^{(r-1)\beta_p}  dx_1 =: c_2 g_{\delta,\eta}(v_1,\ldots,v_p),
\end{align*}
where for the  first equality above, we have integrated out the variables in the order $x_r$,$x_{r-1}$,\ldots, $x_2$ and repeatedly applied  \eqref{eq:beta gamma}. 
Note that the function $g_{\delta,\eta}:[0,1]^p\rightarrow \bb{R}$ is symmetric, so we suppose   without loss of generality that $0\le v_1\le \ldots \le v_p \le 1$ below. Then by monotonicity and \eqref{eq:beta gamma},
\begin{align*}
g_{\delta,\eta}(v_1,\ldots,v_p)=&\int_{\delta}^{\eta}   \left(\prod_{i=1}^p  (x-v_i)^{\beta-1}_+ \right)  (\eta-x)^{(r-1)\beta_p}  dx \le  \int_{\delta}^{\eta}   (x-v_p)_+^{\beta_p-1} (\eta-x)^{(r-1)\beta_p}dx \notag \\\le &
  1_{\{v_p\ge \delta\}} \B(\beta_p,(r-1)\beta_p+1) (\eta-v_p)_+^{r\beta_p}+  1_{\{v_p< \delta\}}\int_{\delta}^{\eta}   (x-\delta )^{\beta_p-1} (\eta-x)^{(r-1)\beta_p}dx  \notag \\
  \le & c (\eta-\delta)^{r\beta_p}.  
\end{align*}
The proof is concluded. 
\end{proof}

\begin{proof}[Proof of Lemma \ref{Lem:local from cover}] 

All conclusions trivially hold if $t=0$. Suppose $0< t\le 1$.

\noindent(a) 
Let 
\[\Theta_*:=  \bigcup_{n=1}^\infty \{\theta:\ V_{1/n}(\theta)=V_0(\theta)\},
\] and we have $\mathbb{P}_V(\Theta_*)=1$ by \eqref{eq:tv coupling}. 
Let \[V_0^{M}(\mbf\theta)= \max(V_0(\theta_i),~i=1,\ldots,p).\]   Set
\[
D_1=\{\mbf{\theta}\in \Theta^p_*:\  0<V_0^{M}(\mbf\theta)<t \}, \quad D_2=\{\mbf{\theta}\in \Theta^p_*:\   V_0^{M}(\mbf\theta)>t\}
\]
and
\[E_1= H^p \times  D_1,\quad   E_2=H^p \times  D_2  ,
\]
 which satisfy $(\mathbb{P}_\cl{N}^p\times \mathbb{P}_V^p)(E_1\cup E_2)  =1$ because the distribution $\wt{\pi}_0$ of $V_0$ is continuous. We shall establish the $L^r$ convergence restricted on $E_1$ and $E_2$  respectively.  
  Let $\E_{\cl{N}}$ and $\E_{V}$ denote  the integrations with respect to   $\mathbb{P}_{\cl{N}}^p$ and $\mathbb{P}_V^p$ respectively.

First, suppose $\mbf{\theta}\in D_1$. In this case, since $V_0^M(\mbf{\theta})\ge \inf \left(\cap_{i=1}^p R_\epsilon(h_i)+V_{0}(\theta_i)\right)$,
\begin{align*}
L_t^{(\epsilon)}(\mbf{h},\mbf{\theta})=\frac{1}{\Gamma(\beta_p)}\left(\frac{\epsilon}{e}\right)^{\beta_p-1} \int_{V_0^M(\mbf{\theta})}^t 1\left\{x\in  \bigcap_{i=1}^p (R_\epsilon(h_i)+V_{\epsilon}(\theta_i)) \right\} dx.
\end{align*}
For $m\in \bb{Z}_+$, set \[\delta_m(\mbf{\theta})= (1-m^{-1})V_0^M(\mbf{\theta})+m^{-1}t,\] which satisfies $0 <V_0^M(\mbf{\theta}) <\delta_m(\mbf{\theta})<t$. Define
\begin{align}\label{eq:L_t eps m}
L_t^{(\epsilon,m)}(\mbf{h},\mbf{\theta})=\frac{1}{\Gamma(\beta_p)} \left(\frac{\epsilon}{e}\right)^{\beta_p-1} \int_{\delta_m(\mbf{\theta})}^t 1\left\{x\in \bigcap_{i=1}^p ( R_\epsilon(h_i)+V_{\epsilon}(\theta_i)) \right\} dx \le L_{t}^{(\epsilon)}(\mbf{h},\mbf{\theta}).
\end{align}
If $\epsilon=1/n$ is sufficiently small so that      $V_\epsilon(\theta_i)=V_0(\theta_i)$ for all $i=1,\ldots,p$ and  $\epsilon<\delta_m(\mbf{\theta})-V_0^M(\mbf{\theta})$,  by \cite[Lemma 2.6]{bai:2019:multstable}, $(L_t^{(\epsilon,m)}(\cdot,\mbf{\theta}))_{\epsilon>0} $ forms a nonnegative $\mathbb{P}_\cl{N}^p$-martingale as $\epsilon$ decreases. So $L_t^{(\epsilon,m)}(\cdot,\mbf{\theta}) $  converges    $\mathbb{P}_{\cl{N}}^p$-a.e.\  as $\epsilon=1/n\rightarrow 0$, and hence  $L_t^{(\epsilon,m)}$ converges $\mathbb{P}_{\cl{N}}^p\times \mathbb{P}_{V}^p$-a.e.\ by Fubini. On the other hand, for any $r\in \bb{Z}_+$, by Fubini, \eqref{eq:bound Delta moment} and \eqref{eq:L_t eps m},
\begin{equation}\label{eq:for ui}
\bb{E}' |L_t^{(\epsilon,m)}|^r \le  \E_{V}\E_\cl{N}[(L_t^{(\epsilon)})^r] \le c t^{r\beta_p}.
\end{equation}
Hence the $L^r$ convergence of  $L_t^{(\epsilon,m)}1_{E_1}$    as $\epsilon=1/n\rightarrow 0$ follows from  uniform integrability. On the other hand, by \eqref{eq:bound Delta moment} again,
\[
 \E_{\cl{N}} |L_{t}^{(\epsilon)}(\cdot,\mbf{\theta})- L_{t}^{(\epsilon,m)}(\cdot ,\mbf{\theta})|^r    \le c   [\delta_m(\mbf{\theta})-V_0^M(\mbf{\theta})]^{r\beta_p}  \le c t m^{-r\beta_p}.
\]
Therefore  
\[
\lim_{m\rightarrow \infty}\sup_{\epsilon>0} \|L_{t}^{(\epsilon)}1_{E_1}- L_{t}^{(\epsilon,m)} 1_{E_1}\|_{L^r}=0.
\]
This together with the $L^r$ convergence of  $L_{t}^{(\epsilon,m)}1_{E_1}$ as $\epsilon=1/n\rightarrow 0$ implies that $L_{t}^{(\epsilon)}1_{E_1}$ is   Cauchy   in $L^r$ as $\epsilon=1/n\rightarrow 0$, and thus converges in $L^r$.

 Next, suppose $\mbf{\theta}\in D_2$. When  $\epsilon$ is small enough so that $V_{\epsilon}(\theta_i)=V_0(\theta_i)$ for all $i=1,\ldots,p$,  we have $L_{t}^{(\epsilon)}(\mbf{h},\mbf{\theta})=0$  
since $t\le \inf\left(\cap_{i=1}^p R_\epsilon(h_i)+V_{0}(\theta_i)\right)$ in this case.  Then the $L^r$ convergence of $L_t^{(\epsilon)}1_{E_2}$ follows from  uniform integrability by \eqref{eq:for ui}.

\medskip
\noindent(b)   
 By   \eqref{eq:f_epsilon}, \eqref{eq:Delta moment pre} and     monotone convergence theorem, we have for $\mbf{\theta}\in \Theta_*^p$  
\begin{align*}
 \E_{\cl{N}} L_t^*(\cdot,\mbf{\theta})^r &= \frac{r!}{\Gamma(\beta_p)^r}\int_{0<x_1<\ldots<x_r<t}  \left(\prod_{i=1}^p  (x_1-V_0(\theta_i))_+^{\beta-1} \right)  (x_2-x_1)^{\beta_p-1} \ldots  (x_r-x_{r-1})^{\beta_p-1}   d\mbf{x}.
\end{align*}
Hence by Fubini and \eqref{eq:beta gamma},
\begin{align*}
\E'[(L_t^*)^r] &=    \frac{r!(1-\beta)^p}{\Gamma(\beta_p)^r}  \int_{0<x_1<\ldots<x_r<t}  \int_{[0,1]^p} \prod_{i=1}^p  v^{-\beta}_i \left(\prod_{i=1}^p  (x_1-v_i)_+^{\beta-1} \right)  d\mbf{v}   \\ & \qquad\qquad\qquad  \times (x_2-x_1)^{\beta_p-1} \ldots  (x_r-x_{r-1})^{\beta_p-1}   d\mbf{x} \\
&=  \frac{r!(1-\beta)^p\B(1-\beta,\beta)^p}{\Gamma(\beta_p)^r}  \int_{0<x_1<\ldots<x_r<t}  (x_2-x_1)^{\beta_p-1} \ldots  (x_r-x_{r-1})^{\beta_p-1} d\mbf{x}.
\end{align*}
The last expression is equal to that in \eqref{eq:Delta moment} through repeated applications of   \eqref{eq:beta gamma}.

\medskip\noindent(c)
 This can be   proved as   \cite[Lemma 2.7]{bai:2019:multstable}, so we only provide a sketch.  Write \[\bigcap_{i=1}^p \left( R_0(h_i)+V_0(\theta_i)\right) =R^*(\mbf{h},\mbf{\theta})+V^*(\mbf{h},\mbf{\theta}),\] 
 where   $V^*(\mbf{h},\mbf{\theta})=\inf\cap_{i=1}^p \left(R_0(h_i)+V_0(\theta_i)\right)$ and $R^*$ is a $\beta_p$-stable regenerative set starting at the origin  independent of $V^*$  under $\mathbb{P}_V^p\times \mathbb{P}_\cl{N}^p$  \cite[Lemma 3.1]{samorodnitsky:2019:extremal}. By its construction,  $(L_t^*)_{t\ge 0}$ is  an
 additive  functional  which only increases on $R^*+V^* $.  In addition, by \eqref{eq:Delta moment}, increment-stationarity of $L_t^*$ and Kolmogorov continuity theorem,  $(L_t^*)_{t\ge 0}$ admits a $\mathbb{P}_\cl{N}^p\times \mathbb{P}_V^p$-version which is continuous in $t$.  Therefore, in view of  \cite[Theorem 3]{kingman:1973:intrinsic} and \cite[Theorem (3.1)]{maisonneuve:1987:subordinators}, the equality in \eqref{eq:L*=L} holds up to a positive multiplicative constant. This constant can be shown to be $1$ by comparing the  moments  as in the proof of \cite[Lemma 2.7]{bai:2019:multstable}.

\end{proof}

\subsection{Proof of Theorem \ref{Thm:Herm Proc}}

\begin{proof} 
We first show the equivalence between the representations in Theorem \ref{Thm:Herm Proc}  (a)  and  (b),   for which  we shall fix $T>0$ and consider  $t\in [0,T]$.  In view of Proposition \ref{Pro:stat set} and Lemma \ref{Lem:change space mult int},  the representation in \eqref{eq:new rep 1} is equivalent to 
\begin{equation}\label{eq:rep interm}
c_{p,\beta}\int_{(\mbf{F}\times [0,\infty))^p}' L_t  \left(\bigcap_{i=1}^p (R_i +v_i)\right)  W^*(d
R_1,dv_1)\ldots W^*(dR_p,dv_p),  
\end{equation}
where $W^*$ is a Gaussian random measure on $\mbf{F}\times [0,\infty)$ with control measure $\bb{P}_R\times \pi_V$. Observe that since $t\le T$, the integrand above is zero if $v_i>T$ for some $i=1,\ldots,p$. So  the integral domain $(\mbf{F}\times [0,\infty))^p$ in \eqref{eq:rep interm} can be replaced by $(\mbf{F}\times [0,T])^p$, and   $\pi_V$ can be viewed as its restriction on $[0,T]$.
Then   define a Gaussian random measure by 
\begin{equation}\label{eq:W_T^*}
W_T^*(\cdot)=\pi_V([0,T])^{-1/2} W^*(\cdot)=T^{(\beta-1)/2} W^*(\cdot),
\end{equation}
 whose control measure is now the probability measure $\bb{P}_R\times \wt{\pi}_V$, where $\wt{\pi}(dv)=T^{\beta-1}(1-\beta) v^{-\beta}dv$, $v\in (0,T)$. Substituting $W^*$ by $W_T^*$ in  \eqref{eq:rep interm}, the equivalence to the representation in (b) then follows from Lemma \ref{Lem:change space mult int}.  Also because of \eqref{eq:W_T^*}, the relation $c_{p,\beta,T}=T^{p(1-\beta)/2}c_{p,\beta}$ holds.

Next we    prove (b). We shall assume $T=1$ for simplicity, and the argument is similar for general $T$. By Lemma \ref{Lem:local from cover},  the $L^2$ linear isometry of multiple Wiener-It\^o integrals (Section \ref{Sec:wi}) and the standardization of variance at $t=1$,  the second moment of the expression in \eqref{eq:rep interm}  when  $t=1$ is equal to   
\[
p!  \frac{2! \Gamma(\beta)^p \Gamma(2-\beta)^p}{\Gamma(\beta_p)\Gamma(\beta_p+2)} c_{p,\beta}^2=1. 
\]
This implies \eqref{eq:c_{p,beta}}.

 Let $(H,\cl{H},\mathbb{P}_\cl{N})$, $(\Theta,\cl{G},\mathbb{P}_V)$, $L_t^*$, $L_t^{(\epsilon)}$, $V_\epsilon$ and $R_\epsilon$ 
be as in Lemma \ref{Lem:local from cover}.  Let   $\nu$ and $U_\epsilon$ be as described in the paragraph above \eqref{eq:stat improper eps}.  In view of Lemma \ref{Lem:change space mult int}, without loss of generality,  one can   set the probability space   in \eqref{eq:new rep 2} as \[(\Omega',\cl{F}',\mathbb{P}')= (H\times \Theta, \cl{H}^{p}\times\cl{G}^p, \mathbb{P}_\cl{N}\times \mathbb{P}_{V}),\]
  assume that $ \mathbb{P}_{\cl{N}}\times \mathbb{P}_{V}$ is atomless, and choose $R=R_0$ and $V=V_0$.
In view of (\ref{eq:L*=L}), for $0 \le  t\le 1$ we have a.s.  
\begin{align*}
\wt{Z}(t):=c_{p,\beta}^{-1}Z(t) = \int_{(H\times\Theta )^p}' L_t^*(\mbf{h},\mbf{\theta}) W(dh_1,d\theta_{1})\ldots  W(dh_p,d\theta_{p}).
\end{align*}
For $\epsilon>0$, define
\begin{equation}\label{eq:H_eps}
\wt{Z}_{\epsilon}(t)=\int_{(H\times\Theta )^p}' L_t^{(\epsilon)}(\mbf{h},\mbf{\theta}) W(dh_1,d\theta_{1})\ldots  W(dh_p,d\theta_{p}).
\end{equation}
By Lemma \ref{Lem:local from cover},    $L_t^{(\epsilon)}$   converges to  $L_t$ in $L^2$ as $\epsilon=1/n\rightarrow 0$.   So by the $L^2$ linear isometry of multiple integrals,  as $\epsilon=1/n\rightarrow 0$,
\begin{equation}\label{eq:L^2 conv}
\wt{Z}_{\epsilon}(t)\overset{L^2(\Omega)}{\longrightarrow}   \wt{Z}(t), \quad t\ge 0.
\end{equation}   

Next, for $\epsilon>0$, define the Gaussian processes 
\begin{equation}\label{eq:Gaussian proc}
G_{\epsilon}(x)= \left(\frac{\pi_\epsilon([0,1])}{ d_\epsilon}\right)^{1/2} \int_{H\times \Theta} 1\{x\in R_\epsilon(h)+V_{\epsilon}(\theta)\}  W(dh,d\theta), \quad x\in [0,1], 
\end{equation}
and 
\begin{equation}
G_{\epsilon}^*(x)= d_\epsilon^{-1/2} \int_{H\times \Theta} 1\{x\in R_\epsilon(h)+U_{\epsilon}(\theta)\}  W(dh,d\theta), \quad x\in [0,\infty).
\end{equation}
We claim that the Gaussian process $(G_{\epsilon}^*(x))_{x\in [0,\infty)}$ is stationary  and
\begin{equation}\label{eq:G=G^*} 
(G_\epsilon(x))_{x\in [0,1]}\EqFDD(G_\epsilon^*(x))_{x\in [0,1]},
\end{equation}
where $\EqFDD$ means equality in finite-dimensional distributions.  
Indeed for $0\le x\le y$, in view of \eqref{eq:stat improper eps},
\begin{align*}
\E [G_\epsilon^*(x)G_\epsilon^*(y)] &=  d_\epsilon^{-1}  \nu(\{x,y\}\subset R_\epsilon+U_\epsilon) =  d_\epsilon^{-1} \nu(\{0,y-x\}\subset R_\epsilon+U_\epsilon )=\E G_\epsilon^*(0)G_\epsilon^*(y-x).
\end{align*}
On the other hand, for $0\le x\le y\le 1$,
\begin{align*}
\E[ G_\epsilon(x)G_\epsilon(y)] &=  (\pi_\epsilon([0,1])/d_\epsilon) \mathbb{P}'(\{x,y\}\subset R_\epsilon+V_\epsilon)=  d_\epsilon^{-1} \nu(\{x,y\}\subset R_\epsilon+U_\epsilon, U_\epsilon \le 1)\\& = d_\epsilon^{-1} \nu(\{x,y\}\subset R_\epsilon+U_\epsilon)=\E G_\epsilon^*(x)G_\epsilon^*(y).
\end{align*}
The reason for introducing $G^*_\epsilon$ in addition to $G_\epsilon$ is because we will apply the spectral representation (see \eqref{eq:spec rep G_eps*} below) of a stationary Gaussian process defined on $\bb{R}$. While  $G^*_\epsilon$ defined on $[0,\infty)$ can be extended to a stationary Gaussian process on $\bb{R}$ by shift, it is not the case for $G_\epsilon$     defined only on $[0,1]$.

Next, because $0\in R_\epsilon$, we have $\{0,x\}\in R_\epsilon+U_\epsilon$ if and only if $U_\epsilon=0$ and $x\in R_\epsilon$. So by (\ref{eq:pi_epsilon}) and \eqref{eq:R_eps U_eps},
\begin{align*}
\E [G_\epsilon^*(0) G_\epsilon^*(x)] =d_\epsilon^{-1}\pi_\epsilon(\{0\})  p_\epsilon(x) =p_\epsilon(x),
\end{align*}
and in particular
$\E G_\epsilon^*(0)^2=p_\epsilon(0)=1$.  Using the spectral representation  \cite[Theorem 7.54]{janson:1997:gaussian}, we have  for $\epsilon>0$ and $x\ge 0$ that
\begin{equation}\label{eq:spec rep G_eps*}
(G_\epsilon^*(x))_{x\in [0,\infty)}\EqFDD\left(\int_{\bb{R}}  e^{i\lambda x} \wh{W}_\epsilon(d\lambda)\right)_{x\in [0,\infty)},
\end{equation}
where $\wh{W}_\epsilon$ is a complex-valued Hermitian Gaussian random measure   with control measure $\mu_\epsilon$ satisfying
\begin{equation}\label{eq:p_eps cov}
p_{\epsilon}(x)=\int_\bb{R} e^{i\lambda x} \mu_\epsilon(d\lambda),\quad  x\ge 0.
\end{equation}   
In addition, using Gaussian moments, we have for some constant $c>0$ not depending on $x$ that   
\[
\E[G^*(x)-G^*(0)]^4=3(\E[G^*(x)-G^*(0)]^2)^2= 12(1-p_{\epsilon}(x))^2\le c x^2,\quad  x\ge 0,
\]
where the inequality follows from an examination of \eqref{eq:point hit R_epsilon}.
Hence $G_\epsilon$ and $G_\epsilon^*$ admit continuous versions 
 by Kolmogorov continuity theorem. We shall work with such versions of them when integrating along their time variables below.

Applying a stochastic Fubini Theorem \cite[Theorem 5.13.1]{peccati:taqqu:2011:wiener} to \eqref{eq:H_eps} (see also \eqref{eq:L_t^eps}), and using the relation between   Hermite polynomials and multiple Wiener-It\^o integrals  \cite[Theorem 7.52,  Equation (7.23) and Theorem 3.19]{janson:1997:gaussian},  \eqref{eq:G=G^*} and \eqref{eq:spec rep G_eps*}, we have  
\begin{align}\label{eq:H mult int}
\wt{Z}_{\epsilon}(t)& \ \, \overset{\text{a.s.}}{=}\frac{1}{\Gamma(\beta_p)}  \left( \frac{\epsilon}{e}\right)^{\beta_p-1} \left(\frac{d_\epsilon}{ \pi_\epsilon([0,1]) } \right)^{p/2} \int_0^t H_p\left(G_{\epsilon}(x)\right)  dx  \notag\\ & \EqFDD \frac{1}{\Gamma(\beta_p)}  \left( \frac{\epsilon}{e}\right)^{\beta_p-1} \left(\frac{d_\epsilon}{ \pi_\epsilon([0,1]) } \right)^{p/2} \int_0^t H_p\left(G_{\epsilon}^*(x)\right)  dx \notag\\
&  \ \, \EqFDD \frac{\left(d_\epsilon \left(\epsilon/e\right)^{\beta-1}\right)^p  }{\Gamma(\beta_p)\pi_\epsilon([0,1])^{p/2}}  \int _{\bb{R}^p}'' \frac{e^{i(\sum_{j=1}^p \lambda_j)t}-1}{i(\sum_{j=1}^p \lambda_j)}  \wt{W}_{\epsilon}(d\lambda_1)\ldots \wt{W}_\epsilon(d\lambda_p),  
\end{align}
where $\wt{W}_\epsilon:=d_\epsilon^{-1/2}\wh{W}_\epsilon$ has   control measure $\wt{\mu}_\epsilon:=d_\epsilon^{-1}\mu_\epsilon$. By   \eqref{eq:p_eps cov} and \eqref{eq:renewal density},
\begin{equation}\label{eq:fourier tran}
\int_{\bb{R}} e^{i\lambda x} \wt{\mu}_\epsilon(d\lambda)=d_\epsilon^{-1}p_{\epsilon}(|x|)=u_\epsilon(|x|).
\end{equation} 
 Note that  in view of \eqref{eq:f_epsilon} and \eqref{eq:d_eps},   as $\epsilon\rightarrow 0$,
 \begin{equation}\label{eq:u_eps}
 u_\epsilon(x)\rightarrow u_0(x)= \frac{ x^{\beta-1}}{\Gamma(\beta)}, \quad x> 0.
 \end{equation} 
 Define
\[
\wt {\mu}_0(d\lambda):=c_\beta |\lambda|^{-\beta} d\lambda, \quad \lambda\neq 0,
\]
where $c_\beta=\frac{2(1-\beta)}{\Gamma(\beta)}\int_0^\infty \sin(y)y^{\beta-2}dy$ is a constant ensuring  the relation 
\begin{equation}\label{eq:fourier trans u_0}
4\int_0^\infty \frac{\sin(ax)}{x} u_0(x)dx=     \wt{\mu}_0([-a,a]), \quad a>0,
\end{equation}
which can be obtained via a change of variable $ax=y$ in the integral above.

 We claim that as $\epsilon\rightarrow 0$, 
\begin{align}\label{eq:H_eps weak conv}
 (\wt{Z}_\epsilon(t))_{0 \le  t\le 1}\ConvFDD  \left(\frac{\Gamma(\beta)^p }{\Gamma(\beta_p) \Gamma(2-\beta)^{p/2}}\int _{\bb{R}^p}'' \frac{e^{i(\sum_{j=1}^p \lambda_j)t}-1}{i(\sum_{j=1}^p \lambda_j)}  \wh{W}_0(d\lambda_1)\ldots \wh{W}_0(d\lambda_p)\right)_{_{0  \le t\le 1}},
\end{align}
where $\ConvFDD$ stands for the convergence of the finite-dimensional distributions, $\wh{W}_\epsilon$ is a complex-valued Hermitian Gaussian random measure with control measure $\wt{\mu}_0$,
 The right-hand side of \eqref{eq:H_eps weak conv} is, up to a constant, the frequency-domain representation of  a  Hermite process \eqref{eq:freq domain} after noticing that $ \wh{W}_0(d\lambda) \EqD c_{\beta}^{-1/2} |\lambda|^{-\beta/2}\wh{W}(d\lambda)$. If \eqref{eq:H_eps weak conv} holds, then in view of \eqref{eq:L^2 conv}, the proof is   concluded.

To show \eqref{eq:H_eps weak conv}, in view of \eqref{eq:d_eps}, \eqref{eq:int h_y}, Cram\'er-Wold  device and \cite[Lemma 3]{dobrushin:major:1979:non} (see also   \cite[Proposition 5.3.6]{pipiras:2017:long}), it suffices to show as $\epsilon\rightarrow 0$, the vague convergence 
\begin{equation}\label{eq:vague conv}
\wt{\mu}_\epsilon(d\lambda)\overset{v}{\rightarrow} \wt {\mu}_0(d\lambda),
\end{equation}   
as well as  
\begin{equation}\label{eq:L^2 tight}
\lim_{A\rightarrow\infty}\limsup_{\epsilon\rightarrow 0}\int_{([-A,A]^p)^c} |k_{t}(\lambda_1,\ldots,\lambda_p)|^2 \wt{\mu}_\epsilon(d\lambda_1)\ldots \wt{\mu}_\epsilon(d\lambda_p)=0,
\end{equation}
where
\[
k_t(\lambda_1,\ldots,\lambda_p):=\int_{0}^t e^{is(\lambda_1+\ldots+\lambda_p)} ds= \frac{\exp(i(\sum_{j=1}^p \lambda_j)t)-1}{i(\sum_{j=1}^p \lambda_j)}, \quad t> 0, \quad \sum_{j=1}^p \lambda_j\neq 0.
\]

We first show \eqref{eq:vague conv}. By  an inversion   of the Fourier transform \eqref{eq:fourier tran}  \cite[Theorem 4:4]{lindgren:2012:stationary},  we have for any $a>0$,
\begin{align*}
\wt{\mu}_\epsilon([-a,a])= \lim_{A\rightarrow\infty}\int_{-A}^A \frac{e^{-iax}-e^{iax}}{-ix} u_\epsilon(x)dx=4\int_{0}^\infty \frac{\sin(ax)}{x} u_\epsilon(x)dx,
\end{align*}
where for the last expression above, its integrability  follows from the fact  (see \eqref{eq:renewal density},  \eqref{eq:d_eps} and \eqref{eq:bound u_epsilon})
\begin{equation}\label{eq:u_eps bound}
u_\epsilon(x)\le c x^{\beta-1},\quad x>0.
\end{equation} 
Note that in the inversion above, we have implicitly used the  continuity of $\wt{\mu}_\epsilon([-a,a])$ in $a$, which  can be verified using the dominated convergence theorem via the bound $\sup_{a\in [0,b]}|\sin(ax)|\le (bx)\wedge 1$ for any $x,b>0$.  Then
by  \eqref{eq:u_eps bound}, \eqref{eq:u_eps}, \eqref{eq:fourier trans u_0} and the dominated convergence theorem, we conclude  as $\epsilon\rightarrow 0$ 
\[
\wt{\mu}_\epsilon([-a,a])\rightarrow \wt{\mu}_0([-a,a]),\] and thus \eqref{eq:vague conv} holds.

To show \eqref{eq:L^2 tight}, define a measure on $\bb{R}^p$ as: 
\begin{equation*} 
\kappa_\epsilon(d\lambda_1,\ldots,d\lambda_p):=|k_t(\lambda_1,\ldots,\lambda_p)|^2 \wt{\mu}_\epsilon(d\lambda_1)\ldots \wt{\mu}_\epsilon(d\lambda_p).  
\end{equation*}

 We shall obtain \eqref{eq:L^2 tight} as a  tightness  condition  from the    weak convergence of   $\kappa_\epsilon$ as $\epsilon\rightarrow 0$.
Indeed, set $\mbf{1}=(1,\ldots,1)\in \bb{R}^p$  and let $\langle \cdot , \cdot \rangle$ denote the Euclidean inner product. By \eqref{eq:fourier tran}, \eqref{eq:u_eps bound}, Fubini and the dominated convergence theorem, we have as $\epsilon\rightarrow 0$
\begin{align*}
\int_{\bb{R}^p} e^{i \langle \mbf{\lambda} ,\mbf{x} \rangle } \kappa_\epsilon(d\mbf{\lambda})&= \int_{\bb{R}^p}  \wt{\mu}_\epsilon^p(d\mbf{\lambda}) e^{i \langle  \mbf{\lambda} , \mbf{x}\rangle } \int_{0}^t e^{is_1 \langle \mbf{\lambda},\mbf{1}\rangle} ds_1\int_{0}^t e^{-is_2 \langle\mbf{\lambda},  \mbf{1}\rangle} ds_2 
\\& =  \int_0^tds_1\int_0^tds_2  \prod_{j=1}^p    u_\epsilon(|x_j +s_1-s_2|) \\
&\rightarrow  \int_0^tds_1\int_0^tds_2  \prod_{j=1}^p    u_0(|x_j +s_1-s_2|)\\&=\frac{t^{\beta_p+1}}{\Gamma(\beta)^{p}}\int_{-1}^{1} (1-|y|)\prod_{j=1}^p |x_j/t+y|^{\beta-1}dy =: \phi(\mbf{x}),
\end{align*} 
  where the last  line is obtained by change of variables $s_1=t(y+w)$, $s_2=t w$ and integrating out the variable $w$.   Note that $\phi(\mbf{x})<\infty$ for any $\mbf{x}\in \bb{R}^p$ since $(\beta-1)p>-1$.     Furthermore,  the function $\phi(\mbf{x})$   is continuous    \cite[Lemma 1]{dobrushin:major:1979:non}  . 
So tightness \eqref{eq:L^2 tight} holds by L\'evy's continuity theorem.
Hence \eqref{eq:H_eps weak conv} is established and the proof is   complete.
 \end{proof}

 \bigskip
 \noindent \textbf{Acknowledgment}. The author would like to thank Takashi Owada and Yizao Wang for discussions which motivate this work, and also thank the anonymous referees for their helpful suggestions.
 
 \bigskip 
 
\noindent 
Shuyang Bai\\
Department of Statistics\\ 
University of Georgia\\
310 Herty Drive, \\
Athens, GA, 30602, USA. \\
{bsy9142@uga.edu}

\end{document}